\documentclass{amsart}

\usepackage{amsmath,graphicx,amsfonts}

\newtheorem{theorem}{Theorem}[section]

\newcommand{\crs}{\overline{\operatorname{cr}}}

\title{Computational search of small point sets with small
rectilinear crossing number.}

\author{Ruy Fabila-Monroy}
\address{\newline Departamento de Matem\'aticas.
  \newline Centro de Investigación y de Estudios Avanzados del Instituto Politécnico Nacional.
 \newline Mexico City, Mexico.} 
\email{ruyfabila@math.cinvestav.edu.mx}

\author{Jorge L\'opez}
\address{\newline Escuela de F\'isica y Matem\'aticas.
  \newline Instituto Polit\'ecnico Nacional.
  \newline Mexico City, Mexico.
} 

\thanks{R.~Fabila-Monroy and J.~L\'opez are supported by CONACyT of Mexico grant 153984.}

\begin{document}


\date{\today}

\maketitle

\begin{abstract}
Let $\crs(K_n)$ be the minimum number of crossings over all
rectilinear drawings of the complete graph on $n$ vertices on the plane.
In this paper we prove that $\crs(K_n) < 0.380473\binom{n}{4}+\Theta(n^3)$;
improving thus on the previous best known upper bound. This is done
by obtaining new rectilinear drawings of $K_n$ for small values
of $n$, and then using known constructions to obtain arbitrarily large good drawings
from smaller ones. The ``small'' sets where found using a simple
heuristic detailed in this paper.
 \end{abstract}

\section{Introduction}

A \emph{rectilinear drawing} of a graph is a drawing
of the graph in the plane in which all the edges are drawn as straight line
segments. For a set $S$ of $n$ points in general position in the plane, 
let $\crs(S)$ be the number of (interior) edge crossings in a
rectilinear drawing of the complete graph $K_n$ with
vertex set $S$. The \emph{rectilinear crossing number} of $K_n$, denoted by $\crs(K_n)$, 
is the minimum of $\crs(S)$ overall all sets of $n$ points
in general position in the plane. The problem of bounding the rectilinear crossing
number of $K_n$, is an important problem in Combinatorial Geometry.
Most of the progress has been made in the last decade, 
for a state-of-the-art survey see \cite{survey}. 
Since two edges cross if and only if their endpoints span 
a convex quadrilateral, $\crs(S)$ is equal to the number $\square(S)$,
of convex quadrilaterals spanned by $S$. We use this equality extensively through out the paper.
The current best bounds for $\crs(K_n)$ are~\cite{lower,upper}:

\[ 0.379972 \binom{n}{4} < \crs(K_n)  < 0.380488\binom{n}{4}+\Theta(n^3) \]

Our main result is the improvement of the upper bound to:

\begin{theorem}\label{thm:main}
\[ \crs(K_n) \le \frac{9363184}{24609375}\binom{n}{4}+\Theta(n^3) < 0.380473\binom{n}{4}+\Theta(n^3) \]
\end{theorem}

Although it is a modest improvement, we note that the gap between the
lower and upper bound is already quite small and that actually 
the lower bound is conjectured to be at least $0.380029\binom{n}{4}+\Theta(n^3)$.
In~\cite{upper} the authors conjecture that every optimal set 
is $3$-decomposable\footnote{$S$ is $3$-decomposable if there is a triangle
$T$ enclosing $S$, and a balanced partition $(A,B,C)$ of $S$, 
such that the orthogonal projections of $S$ onto the sides of $T$ show $A$ between $B$ and $C$ on one
side, $B$ between $A$ and $C$ on another side, and $C$ between $A$ and $B$ on the third side.}, and
show that every $3$-decomposable set contains at least $0.380029\binom{n}{4}+\Theta(n^3)$ crossings.
The current general approach to produce rectilinear drawings
of $K_n$ with few crossings, is to start with a 
drawing with few crossings (for a small value of $n$), and use it to recursively
obtain drawings with few number of crossings for arbitrarily 
large values of $n$. This approach has been refined
and improved over the years~\cite{singer,towards,oswinupper,bersil,upper}.

The upper bound provided by the best recursive construction to this date
is expressed in Theorem~\ref{thm:cons}.

\begin{theorem}\label{thm:cons}\emph{(Theorem~3 in \cite{upper})}
If $S$ is an $m$-element point set in general position, with $m$ odd, then
$$\crs(K_n)\le \frac{24\crs(S)+3m^3-7m^2+(30/7)m}{m^4}\binom{n}{4}+\Theta(n^3)$$
\end{theorem}

Given these recursive constructions, there is a natural interest
in finding sets with few crossings for small values of $n$.
The use of computers to aid this search was initiated in 
\cite{database}.

\section{Results}

For $n \le 100$, we improved many of best known point sets
of $n$ points with few crossings using the following simple heuristic.

Given a starting set $S$ of $n$ points in general position in the plane, do:
\begin{itemize}
 \item {\bf Step 1.} Choose randomly a point $p\in S$.
 \item {\bf Step 2.} Choose a random point $q$ in the plane ``close'' to $p$.
 \item {\bf Step 3.} If $\crs(S\setminus \{p\} \cup \{q\}) \le \crs(S)$,
                     then update $S$ to $S:=S\setminus \{p\} \cup \{q\}$. 
 \item {\bf Step 4.} Go to Step 1.
    
\end {itemize}

For each $n=3,\dots,100$, as a starting set we chose the best known point set;
they can be downloaded from Oswin Aichholzer's homepage at: \\ {\tt www.ist.tugraz.at/aichholzer/research/rp/triangulations/crossing/ }\\
In many instances we managed to improve the previous best examples. 
Our results are shown in Table~\ref{tbl:ours}. Theorem~\ref{thm:main} now follows directly from Theorem~\ref{thm:cons} using the set of
$75$ points we found with $450492$ crossings.

\begin{table}[t]
\caption{Improvements on previous point sets.}
\label{tbl:ours}
\centering
\begin{tabular}{|c|c|c |c|c|c|c|}

\cline{1-3} \cline{5-7}
 $n$ & Number of      & Number of      & & $n$ & Number of      & Number of      \\
     & crossings in   & crossings in   & &     & crossings in   & crossings in   \\
     & the best point & the previous   & &     & the best point & the previous   \\         
     & set obtained   & best point set & &     & set obtained   & best point set \\
\cline{1-3} \cline{5-7}
\cline{1-3} \cline{5-7}
46 & 59463 & 59464 &  & 76 & 475793 & 475849 \\
\cline{1-3} \cline{5-7}
47 & 65059 & 65061 &  & 77 & 502021 & 502079 \\
\cline{1-3} \cline{5-7}
49 & 77428 & 77430 &  & 78 & 529291 & 529350 \\
\cline{1-3} \cline{5-7}
50 & 84223 & 84226 &  & 79 & 557745 & 557849 \\
\cline{1-3} \cline{5-7}
52 & 99169 & 99170 &  & 80 & 587289 & 587367 \\
\cline{1-3} \cline{5-7}
53 & 107347 & 107355 &  & 81 & 617958 & 618048 \\
\cline{1-3} \cline{5-7}
54 & 115979 & 115994 &  & 82 & 649900 & 649983 \\
\cline{1-3} \cline{5-7}
56 & 134917 & 134930 &  & 83 & 682986 & 683096 \\
\cline{1-3} \cline{5-7}
57 & 145174 & 145178 &  & 84 & 717280 & 717384 \\
\cline{1-3} \cline{5-7}
58 & 156049 & 156058 &  & 85 & 753013 & 753079 \\
\cline{1-3} \cline{5-7}
59 & 167506 & 167514 &  & 86 & 789960 & 790038 \\
\cline{1-3} \cline{5-7}
61 & 192289 & 192293 &  & 87 & 828165 & 828233 \\
\cline{1-3} \cline{5-7}
63 & 219659 & 219683 &  & 88 & 867911 & 868023 \\
\cline{1-3} \cline{5-7}
64 & 234464 & 234470 &  & 89 & 908972 & 909128 \\
\cline{1-3} \cline{5-7}
65 & 249962 & 249988 &  & 90 & 951418 & 951526 \\
\cline{1-3} \cline{5-7}
66 & 266151 & 266188 &  & 91 & 995486 & 995678 \\
\cline{1-3} \cline{5-7}
67 & 283238 & 283286 &  & 92 & 1040954 & 1041165 \\
\cline{1-3} \cline{5-7}
68 & 301057 & 301098 &  & 93 & 1087981 & 1088217 \\
\cline{1-3} \cline{5-7}
69 & 319691 & 319737 &  & 94 & 1136655 & 1136919 \\
\cline{1-3} \cline{5-7}
70 & 339254 & 339297 &  & 95 & 1187165 & 1187263 \\
\cline{1-3} \cline{5-7}
71 & 359645 & 359695 &  & 96 & 1238918 & 1239003 \\
\cline{1-3} \cline{5-7}
72 & 380926 & 380978 &  & 97 & 1292796 & 1292802 \\
\cline{1-3} \cline{5-7}
73 & 403180 & 403234 &  & 98 & 1348070 & 1348072 \\
\cline{1-3} \cline{5-7}
74 & 426419 & 426466 &  & 99 & 1405096 & 1405132 \\
\cline{1-3} \cline{5-7}
75 & 450492 & 450550 & & & & \\
\cline{1-3} \cline{5-7}
\end{tabular}
\end{table}

\section{The Algorithm}\label{sec:alg}
In this section we describe an $O(n^2)$ time algorithm 
used to compute $\crs(S)$ in step 3 of the
heuristic. Recall that $\crs(S)$ is equal to $\square(S)$. 
We compute this number instead.
Quadratic time algorithms for computing $\square(S)$ have been known 
for a long time~\cite{rote1,rote2}.
We learned of these algorithms after we finished 
the implementation of our algorithm. We present our algorithm nevertheless, 
since in the process we obtained an equality(Theorem~\ref{thm:patterns})
between certain substructures of $S$ and $\crs(S)$, which
may be of independent interest.
We also think that given that the main aim of this
paper is to communicate the method by which we obtained these sets, it is important
to provide as many details as possible so that an interested reader
can obtain similar results.

We compute $\square(S)$ by computing the number of certain
subconfigurations of $S$ which determine
$\square(S)$. Let $(p,q)$ be an ordered pair of distinct points in $S$, and let
$\{r,s\}$ be a set of two points of $S \setminus \{p,q\}$.
We call the tuple $((p,q),\{r,s\})$ a \emph{pattern}.
We say that $((p,q),\{r,s\})$ is of \emph{type} $A$ if $q$ lies in
the convex cone with apex $p$ and bounded by the rays $\overrightarrow{pr}$
and  $\overrightarrow{ps}$, otherwise it is of \emph{type} $B$.
Let $A(S)$ be the number of type $A$ patterns in $S$, and $B(S)$ the number
of its type $B$ patterns. Note that every choice of  $((p,q),\{r,s\})$
is either an $A$ pattern or a $B$ pattern. The number of these
patterns determine $\square(S)$ as the following Theorem
shows.

\begin{theorem}\label{thm:patterns}
\[ \square(S)=\frac{3A(S)-B(S)}{4}\]
\end{theorem}
\begin{proof}
Let $X$ be a subset of $S$, of $4$ points.
Simple arithmetic shows that if $X$ is not in convex position
then it determines $3$ patterns of type $A$ and $9$ patterns
of type $B$; on the other hand if $X$ is in convex position
then it determines $4$ patterns of type $A$
and $8$ patterns of type $B$. Assume that we assign a value
of $3$ to type $A$ patterns and a value of $-1$ to 
type $B$ patterns. If $X$ is not in convex
position its total contributed value would be zero and if it
is convex position it would be $4$.
Thus $4\square(S)=3A(S)-B(S)$, and the result follows.
\end{proof}

Note that the total number patterns is $n(n-1)\binom{n-2}{2}$.
Thus by Theorem~\ref{thm:patterns} to compute $\square(S)$ 
it is sufficient to compute $A(S)$. Let $p$ be a point in $S$. We now show how to count the number of type $A$ patterns
in which $p$ is the apex of the corresponding wedge.

Sort the points in $S\setminus \{ p \} $
counterclockwise by angle around $p$. Let $y_1,y_2,\dots,y_{n-1}$
be these points in such an order. 
For $1 \le i \le n-1$, starting from $y_i$ and going counterclockwise, 
let $i'$ be the first index (modulo $n$) such that the angle $\angle y_ipy_{i'}$ 
is more than $\pi$. Let $m_i:=i'-i \mod (n-1)$. 
Note that for $1 \le  j \le m_i$ there are  exactly $j-1$ 
type $A$ patterns of the form $(p,q),\{y_i,y_{i+j}\}$ for some $q\in S$. In total,
summing over all such $j$'s, this amounts to $\sum_{j=1}^{m_i}(j-1)=\binom{m_i}{2}$.
Thus the total number of type $A$ patterns
in which $p$ is the apex of the corresponding wedge is
equal to $\sum_{i=1}^{n-1}\binom{m_i}{2}$.

Compute $y_{1'}$ and $m_1$ from scratch in linear time. For $2\le i \le n-1$,  to compute
$y_{(i+1)'}$ and $m_{i+1}$, assume that we have computed $y_{i'}$ and $m_i$.
Start from $y_{i'}$ and go counterclockwise
until the first $y_{(i+1)'}$ is found such that 
the angle $\angle y_ipy_{(i+1)'}$ is more than $\pi$; 
then $m_{i+1}=(i+1)-(i+1)'$. Since one pass is done over each
$y_{i'}$, this is done in $O(n)$ total time. Finally,
sorting $S\setminus{p}$ by angle around $p$, for all
$p \in S$, can be done in $O(n^2)$ total time. This is done
by dualizing $S$ to a set of $n$ lines. The corresponding
line arrangement can be constructed in time $O(n^2)$ with standard
algorithms. The orderings around each point can then be 
extracted from the line arrangement in $O(n^2)$ time.

\section{Implementation}

In this section we provide relevant information of the
implementation of the algorithm described in Section~\ref{sec:alg}
and of the searching heuristic we used to obtain the point
sets of Table~\ref{tbl:ours}.

Instead of sorting in $O(n^2)$ time the points
by angle around each point of $S$, we used
standard sorting functions. 
This was done because these functions have been quite optimized, 
and the known algorithms to do it in $O(n^2)$ time are not straight
forward to implement. Thus our implementation actually runs in
$O(n^2\log n)$ time.

All our point sets have integer coordinates.
This was done to ensure the correctness of the computation.
The only geometric primitive involved in the algorithm
 is to test whether certain angles are greater than $\pi$; this
can be done with a determinant. Therefore
as long as all the points have integer coordinates, the result
is an integer as well. We did two implementations of 
our algorithm, one in Python and the other in C. 
In Python, integers have arbitrarily large precision, so the 
Python implementation is always correct.
In the C implementation we used $128$-bit integers. Here, we have
to establish a safety margin---as long
as the absolute value of the coordinates
is at most $2^{62}$, the C implementation will produce a correct
answer. Empirically we observed a $30\times$ speed up of the C implementation
over the Python implementation. At each step of the heuristic
we checked if it was safe to use the (faster) C implementation.

To find the point $q$ replacing $p=(x,y)$ in \textbf{Step 2},
we first chose two natural numbers $t_x$ and $t_y$.
These number were distributed exponentially with a prespecified
mean $M$ and rounded to the nearest integer. Afterwards
with probability $1/2$ they were replaced by their negative.
Point $q$ was then set to $(x+t_x,y+t_y)$. After choosing
an initial mean, the heuristic was left to run for some time, if
no improvement was found by then, the mean was halved (or rather
the point set was doubled by multiplying each of its points by two).

All the code used in this paper and the point sets obtained are available 
upon request from the first author. All the point sets obtained
are available for download at the ArXiv page of this paper.

\textbf{Acknowledgments.\\}
We thank Jes\'us Lea\~nos and Gelasio Salazar for various helpful 
discussions.

\bibliographystyle{alpha} \bibliography{crossingbib}

\section*{Set of 75 points with 450492 crossings}
\[
\begin{matrix}
 (4473587539, 8674070321),&(2195118038, 12138376393),&(3359570710, 10389672946)\\ 
(2067188794, 12364750532),&(3798074340, 9176659177),&(-495951185, 16620108498)\\ 
(1133302705, 13923635114),&(1044611367, 14069644578),&(-311149395, 16314077753)\\ 
(2027617952, 3459524378),&(4601468259, 7662169961),&(4601078091, 7662133857)\\ 
(4113182393, 7619250691),&(4116054424, 7605654413),&(3570685582, 9808713565)\\ 
(3722340414, 9316231785),&(4112078622, 7625130881),&(4107912992, 7542476726)\\ 
(4106745227, 7535480343),&(3189483730, 5743999450),&(3168421193, 5701152359)\\ 
(8944839519, 7965414411),&(3955068845, 6639763085),&(4012346331, 6733970340)\\ 
(3648786718, 6305728855),&(3653540692, 6310524663),&(3253433517, 5873175144)\\ 
(2113073755, 12281280867),&(-1364755153, -2899618565),&(1679455404, 2812631891)\\ 
(1549775961, 2575359287),&(2154725117, 3676030999),&(2297590336, 3930708704)\\ 
(1474528964, 2436685704),&(1293365372, 2095165431),&(5207789612, 7710691788)\\ 
(1889666524, 3220648103),&(1902363904, 3245131307),&(4899124137, 8128629846)\\ 
(4897948559, 8128714256),&(5216754785, 7718023020),&(1683153691, 13003463181)\\ 
(5202684700, 7706307614),&(5277878757, 7741749531),&(5279252153, 7742686707)\\ 
(7370957968, 7863465953),&(7493305742, 7871610457),&(3571434484, 9806112525)\\ 
(6168237700, 8065376268),&(6032867454, 8070589271),&(5981198967, 8072572208)\\ 
(6888712646, 7936512772),&(6851478487, 7943849321),&(3214935430, 10605538217)\\ 
(7338699912, 7861922951),&(9000883017, 7965096231),&(4059850707, 6811671897)\\ 
(8806696260, 7963533399),&(3839573186, 9100031657),&(4471841261, 8674882284)\\ 
(15041590733, 8118237065),&(10588618608, 8002947798),&(10174892708, 7993197449)\\ 
(1902291407, 12661152660),&(1811935937, 12802330604),&(11185824774, 8018462436)\\ 
(10634751909, 8004278071),&(9630596054, 7968154616),&(9350903224, 7955792213)\\ 
(4338851382, 8157414467),&(4338568456, 8157953847),&(4520171724, 8637506721)\\ 
(4532317105, 8633237970),&(4538689274, 8630906861),&(3400009645, 10327277784)\\ 
\end{matrix}
\]
\end{document}